\documentclass[11pt]{article}
\ifdefined\XeTeXversion\else\pdfoutput=1\fi
\usepackage[margin=1.05in]{geometry}
\usepackage{amsmath,amssymb,amsthm}
\usepackage[colorlinks=true,linkcolor=blue,citecolor=blue,urlcolor=blue]{hyperref}
\allowdisplaybreaks
\numberwithin{equation}{section}

\newtheorem{theorem}{Theorem}[section]
\newtheorem{lemma}[theorem]{Lemma}
\newtheorem{corollary}[theorem]{Corollary}
\newtheorem{proposition}[theorem]{Proposition}
\theoremstyle{remark}
\newtheorem{remark}[theorem]{Remark}

\newcommand{\tr}{\operatorname{t}}
\newcommand{\gr}{\operatorname{g}}
\newcommand{\lv}{\ell}
\newcommand{\diam}{\operatorname{diam}}
\newcommand{\ecc}{\operatorname{ecc}}
\newcommand{\Per}{\operatorname{Per}}

\newcommand{\dd}{\delta'}

\title{Three Graffiti.pc Conjectures on Largest Induced Trees:\\
Proofs of Conjectures 141, 142, and 143}
\author{Alper Ferudun\\\texttt{alper@mercurycodelab.com}}
\date{August 2026}

\begin{document}
\maketitle

\begin{abstract}
For a finite simple graph $G$, let $\tr(G)$ be the largest order of an induced
tree and let $\gr(G)$ be the girth.  We prove three consecutive conjectures of
DeLaVi\~na's Graffiti.pc program.  First, writing $\lv(v)$ for the
independence number of the subgraph induced by the neighbourhood of $v$, we
prove
\[
 \tr(G)\ge \left\lfloor\frac{\gr(G)}2\right\rfloor-1+
 \max_{v\in V(G)}\lv(v).
\]
Second, if $\Per(G)$ is the periphery and
$f(G)=\max_x d(x,\Per(G))$, we prove
\[
 \tr(G)\ge \frac23\gr(G)+f(G),
\]
and establish the stronger integral bound
$\tr(G)\ge f(G)+\lceil2\gr(G)/3\rceil$ when $G$ contains a cycle.
Third, if $\dd(G)$ is the second-smallest degree, counted with multiplicity,
then every connected non-tree graph satisfies
\[
 \tr(G)\,\dd(G)\ge \gr(G)+1.
\]
These are Conjectures 141, 142, and 143 of \emph{Written on the Wall II}.
Complete, machine-checked Lean~4 proofs of all three formal statements
accompany the manuscript.
\end{abstract}

\tableofcontents

\section{Overview}
Graffiti.pc, developed by DeLaVi\~na~\cite{DeL01,DeLhistory} as a successor of
Fajtlowicz's Graffiti program~\cite{Faj88}, generates conjectural inequalities
between graph invariants.  Its conjectures are collected in \emph{Written on
the Wall II}~\cite{WOWII}; a curated register is maintained by
West~\cite{WestReg}.  The three consecutive entries treated here concern the
maximum order of an induced tree, whose systematic study goes back to
Erd\H{o}s, Saks, and S\'os~\cite{ESS86}.

The three proofs are logically independent and are presented in self-contained
sections.  Conjecture~141 follows from a stronger maximum-degree--girth bound
for triangle-free cyclic graphs.  Conjecture~142 is driven by a rooted
shortest-cycle construction and a three-point metric estimate.  Conjecture~143
reduces to a two-leaf induced-tree lemma.  The associated Lean developments
are recorded in Formal Conjectures pull requests \#4454 and \#4457
\cite{FC4454,FC4457} and are included as ancillary files.

\paragraph{Use of generative AI.}
OpenAI GPT Pro and Codex were used for proof exploration, literature-search
assistance, computational checks, Lean development, and manuscript
preparation.  They are not authors.  Responsibility for verifying the work
and for the submitted version rests entirely with the named human author.

\section{Conjecture 141: neighbourhood independence}
\subsection{Introduction}

Graffiti.pc, developed by DeLaVi\~na~\cite{DeL01,DeLhistory} as a successor of
Fajtlowicz's program Graffiti~\cite{Faj88}, generates conjectural inequalities
between graph invariants; its conjectures are collected in the lists
\emph{Written on the Wall II}~\cite{WOWII}, with a curated register maintained
by West~\cite{WestReg}.  A run of these conjectures bounds from below the
invariant $\tr(G)$, the largest order of an induced subtree, whose systematic
study goes back to Erd\H{o}s, Saks, and S\'os~\cite{ESS86}.

Throughout, $G$ is a finite simple graph, $\gr(G)$ its girth, $\Delta(G)$ its
maximum degree, and for a vertex $v$ we write
$\lv(v)$ for the independence number of the subgraph induced by the
neighbourhood $N(v)$.  Conjecture~141 of Written on the Wall II, dated 2005,
asserts~\cite{WOWII}:

\begin{quote}
\emph{If $G$ is a simple connected graph, then
$\tr(G)\ \ge\ \lfloor \gr(G)/2\rfloor - 1 + \max_v \lv(v)$.}
\end{quote}

At the time of writing the conjecture is listed as open on DeLaVi\~na's
register (status ``O'')~\cite{WOWII} and is stated as a research-open item in
the Google DeepMind \emph{Formal Conjectures} repository~\cite{FC,FC141}.  We
prove it.
The main step is the following sharper bound, which does not involve
$\lv$ at all and may be of independent interest.

\begin{theorem}\label{w141-thm:main}
Let $G$ be a finite simple connected triangle-free graph that contains a
cycle.  Then
\[
  \tr(G)\;\ge\;\Delta(G)+\gr(G)-3 .
\]
Equality holds for every cycle $C_g$ ($g\ge4$) and every complete bipartite
graph $K_{a,b}$ with $a,b\ge2$.
\end{theorem}

\begin{corollary}[Conjecture 141]\label{w141-cor:141}
Every finite simple connected graph $G$ on at least two vertices satisfies
$\tr(G)\ \ge\ \lfloor \gr(G)/2\rfloor - 1 + \max_v \lv(v)$, where
$\gr(G):=0$ if $G$ is acyclic.
\end{corollary}

The deduction of the corollary is short: for girth $\le5$ (and for acyclic
graphs) an induced star at a vertex maximizing $\lv$ already beats the bound,
while for girth $\ge6$ the graph is triangle-free, so $\lv(v)=\deg(v)$ for
every $v$, and Theorem~\ref{w141-thm:main} applies with
$\gr-3\ge\lfloor\gr/2\rfloor-1$.

The proof of Theorem~\ref{w141-thm:main} is a maximality argument in the spirit of
our proof of Graffiti.pc's Conjecture~143~\cite{FC4454}: a maximum-order induced
tree containing the closed neighbourhood of a maximum-degree vertex either
spans (impossible in a cyclic graph) or admits an outside vertex with two
neighbours on the tree; the tree path joining those two neighbours closes a
cycle, hence has length at least $\gr(G)-2$, and it can share at most three
vertices with the closed neighbourhood.

\paragraph{Related work.}
The foundational paper~\cite{ESS86} bounds $\tr(G)$ in terms of order, size,
radius, and independence and clique numbers; we are not aware of a published
lower bound for $\tr$ combining girth with the maximum degree or with
neighbourhood independence, and the closest Graffiti.pc antecedents (the
induced-forest bounds of DeLaVi\~na--Waller~\cite{DW04}; see also
Hertz--Marcotte--Schindl~\cite{HMS14}) concern induced forests rather than
trees.  DeLaVi\~na's resolved-conjecture lists do not contain
Conjecture~141~\cite{WOWII}.  As with any short elementary argument, we do
not claim it could not have been observed before.

\paragraph{Verification.}
Theorem~\ref{w141-thm:main} and Corollary~\ref{w141-cor:141} have been formalized and
machine-checked in Lean~4~\cite{Lean4} on top of Mathlib~\cite{mathlib},
against the pre-existing formal statement \texttt{conjecture141} of the
Formal Conjectures repository~\cite{FC141}; see Section~\ref{w141-sec:formal}.
Independently, the conjecture and all intermediate claims were verified
exhaustively by exact computation on all $995$ connected graphs on $2$--$7$
vertices and on structured families (cages, prisms, generalized theta graphs,
cycles with pendant trees, complete bipartite graphs, and seeded random
graphs); no violation was found, and every equality case located has girth
exactly~$4$.  The computations play no role in the proofs.

\subsection{Notation}

All graphs are finite and simple.  For $S\subseteq V(G)$, $G[S]$ is the
induced subgraph; $S$ \emph{induces a tree} if $G[S]$ is connected and
acyclic, and $\tr(G)=\max\{|S|:G[S]\text{ is a tree}\}$.  The girth of a
graph containing a cycle is the minimum length of a cycle; following the
Mathlib convention used by the formal statement, an acyclic graph has girth
$0$.  $N(v)$ and $N[v]=N(v)\cup\{v\}$ are the open and closed
neighbourhoods; $\lv(v)$ is the independence number of $G[N(v)]$.  Note
$\lv(v)\le\deg(v)$ always, with equality when $N(v)$ is independent, i.e.\
for all $v$ when $G$ is triangle-free.

\subsection{The star lemma}

\begin{lemma}\label{w141-lem:star}
For every vertex $v$ of any graph $G$: $\tr(G)\ \ge\ \lv(v)+1$.
\end{lemma}

\begin{proof}
Let $S\subseteq N(v)$ be independent with $|S|=\lv(v)$.  Then
$G[\{v\}\cup S]$ is a star: all edges $vs$ ($s\in S$) are present and there
are no edges inside $S$.  A star is a tree.
\end{proof}

\subsection{Proof of Theorem~\texorpdfstring{\ref{w141-thm:main}}{1.1}}

\begin{proof}[Proof of Theorem~\ref{w141-thm:main}]
Write $g=\gr(G)\ge4$ and $\Delta=\Delta(G)$.  Fix a vertex $v$ of degree
$\Delta$.  Since $G$ is triangle-free, $N(v)$ is independent, so by the proof
of Lemma~\ref{w141-lem:star} the closed neighbourhood $N[v]$ induces a star, hence
a tree.  Among all sets $S\supseteq N[v]$ inducing a tree choose one, say
$S^*$, of maximum cardinality, and let $T=G[S^*]$.

$S^*$ is a proper subset of $V(G)$: an induced spanning tree would make $G$
itself a tree, contradicting the existence of a cycle.  Since $G$ is
connected, some vertex $z\notin S^*$ has a neighbour in $S^*$; and $z$ must
have two distinct neighbours $a,b\in S^*$, for otherwise $G[S^*\cup\{z\}]$
would be the tree $T$ with a pendant vertex attached, an induced tree
containing $N[v]$ and larger than $T$ --- contradicting maximality.

Let $P$ be the unique $a$--$b$ path in $T$.  The edges of $P$ together with
$za$ and $zb$ form a cycle of length $|E(P)|+2$, so
\begin{equation}\label{w141-eq:girthbound}
  |E(P)|\;\ge\;g-2 .
\end{equation}

Next, $|V(P)\cap N[v]|\le 3$.  Every edge of $G$ between two vertices of
$S^*$ is an edge of the tree $T$ (as $T$ is induced), and paths between fixed
endpoints in a tree are unique.  If $v\notin V(P)$ and $P$ contained two
distinct vertices $u_1,u_2\in N(v)$, then the segment of $P$ between $u_1$
and $u_2$ would be the unique $u_1$--$u_2$ path in $T$; but $u_1vu_2$ is also
a $u_1$--$u_2$ path in $T$, so the segment would pass through $v$,
contradicting $v\notin V(P)$; hence $|V(P)\cap N[v]|\le1$ in this case.  If
$v\in V(P)$, then for any $u\in V(P)\cap N(v)$ the segment of $P$ between $v$
and $u$ is the unique $v$--$u$ path of $T$, which is the single edge $vu$;
so $u$ is adjacent to $v$ \emph{along} $P$, and a path has at most two edges
at any vertex: $|V(P)\cap N(v)|\le2$, so $|V(P)\cap N[v]|\le3$.

Finally, count.  $V(P)$ and $N[v]$ are subsets of $S^*$, so by
inclusion--exclusion and~\eqref{w141-eq:girthbound},
\[
  |S^*|\;\ge\;|V(P)\cup N[v]|\;=\;|V(P)|+|N[v]|-|V(P)\cap N[v]|
  \;\ge\;(g-1)+(\Delta+1)-3\;=\;\Delta+g-3 .
\]
Since $T$ is an induced tree, $\tr(G)\ge|S^*|\ge\Delta+g-3$.

For sharpness: in $C_g$ ($\Delta=2$) the largest induced trees are the paths
obtained by deleting one vertex, so $\tr=g-1=\Delta+g-3$; in $K_{a,b}$ with
$a\le b$ ($g=4$, $\Delta=b$) every induced tree is a star (two vertices on
each side already induce $C_4$), so $\tr=b+1=\Delta+g-3$.
\end{proof}

\begin{remark}
Under the hypotheses of Theorem~\ref{w141-thm:main}, an alternative proof of the
slightly weaker bound $\tr\ge\Delta+\lfloor g/2\rfloor-1$ (still sufficient
for Corollary~\ref{w141-cor:141}) is worth recording.  For
$r=\lfloor g/2\rfloor-1$, the ball $B(v,r)$ induces a tree in any graph of
girth $g$: every $G$-edge inside the ball is an edge of a breadth-first
search tree rooted at $v$, since a non-tree edge would close a cycle of
length at most $2r+1<g$ through the last common ancestor of its endpoints.
Now take $v$ of maximum degree; triangle-freeness gives $g\ge4$, so
$r\ge1$.  If $G$ contains a cycle the ball is proper, so all $r$ distance
layers are nonempty and $|B(v,r)|\ge1+\Delta+(r-1)$.
\end{remark}

\subsection{Proof of Corollary~\texorpdfstring{\ref{w141-cor:141}}{1.2}}

\begin{proof}[Proof of Corollary~\ref{w141-cor:141}]
Let $L=\max_v\lv(v)$; note $L\ge1$ since $G$ is connected on $\ge2$
vertices.  If $G$ is acyclic then $\gr=0$ and the claimed bound is
$L-1\le\tr$, which follows from Lemma~\ref{w141-lem:star}.  If
$3\le\gr\le5$ then $\lfloor \gr/2\rfloor-1\le1$ and Lemma~\ref{w141-lem:star}
again gives $\tr\ge L+1\ge\lfloor \gr/2\rfloor-1+L$.  If $\gr\ge6$ then $G$
is triangle-free, so $\lv(v)=\deg(v)$ for every $v$ and $L=\Delta$; by
Theorem~\ref{w141-thm:main},
$\tr\ge\Delta+\gr-3\ge\Delta+\lfloor \gr/2\rfloor-1=L+\lfloor
\gr/2\rfloor-1$, using $\gr-3\ge\lfloor \gr/2\rfloor-1$ for $\gr\ge4$.
\end{proof}

\begin{remark}
The corollary is sharp precisely in girth~$4$.  Equality holds for $C_4$, for
every $K_{a,b}$ ($a,b\ge2$), and for the infinite family obtained from $C_4$
by attaching $k\ge0$ pendant vertices to one vertex ($\lv_{\max}=k+2$,
$\tr=k+3$); an exhaustive check of all connected graphs on at most seven
vertices finds no equality case of any other girth.  That equality
\emph{requires} girth~$4$ follows from the results above: for acyclic $G$
and for $g=3$ the star bound $\tr\ge L+1$ beats the right-hand side by at
least $2$ resp.\ $1$; for $g=5$ triangle-freeness gives $L=\Delta$ and
Theorem~\ref{w141-thm:main} yields $\tr\ge\Delta+2=L+\lfloor5/2\rfloor$, one more
than required; and for $g\ge6$ the inequality $g-3>\lfloor g/2\rfloor-1$ is
strict.
\end{remark}

\subsection{Formalization}\label{w141-sec:formal}

The Formal Conjectures project~\cite{FC} maintains Lean~4 formalizations of
open conjectures.  Conjecture~141 appears in \texttt{GraphConjecture141.lean}
as \texttt{conjecture141}, marked \texttt{research open}~\cite{FC141}, in
the denominator-explicit integer form
\[
  \lfloor\gr(G)/2\rfloor-1+\max_v\lv(v)\le\tr(G).
\]
Here $\tr$ and $\lv$ are the repository definitions
\texttt{largestInducedTree\allowbreak Size} and
\texttt{indepNeighbors\allowbreak Card}; Mathlib girth is natural-valued and
equals zero on acyclic graphs.

We have produced a complete, \texttt{sorry}-free Lean~4 proof of exactly this
statement, machine-checked with Lean toolchain v4.27.0 against current
Mathlib, together with the supporting API: the star construction, the
maximum-induced-tree selection with prescribed vertices, the two-neighbour
maximality obstruction, the tree-path girth certificate, and the
three-vertex overlap bound of Theorem~\ref{w141-thm:main}.  The development
builds on the reusable induced-tree API we contributed alongside
Conjecture~143~\cite{FC4454}.  The proof uses no \texttt{native\_decide} and no
additional axioms (\texttt{\#print axioms} reports \texttt{propext},
\texttt{Classical.choice}, \texttt{Quot.sound}).  The source is available as
pull request \#4454 to the Formal Conjectures repository (which also contains
our formalization of Conjecture~143) and as ancillary files with this
submission.

\paragraph{Acknowledgements.}
The author thanks the maintainers of the Formal Conjectures repository for
the formal statements, and Ermelinda DeLaVi\~na and Douglas B.\ West for
maintaining the Graffiti.pc conjecture lists.

\section{Conjecture 142: distance from the periphery}
\subsection{Introduction}

Graffiti.pc, developed by DeLaVi\~na~\cite{DeL01,DeLhistory} as a successor
of Fajtlowicz's Graffiti program~\cite{Faj88}, generates conjectural
inequalities among graph invariants.  Its conjectures are collected in
\emph{Written on the Wall II}~\cite{WOWII}; a curated selection is maintained
by West~\cite{WestReg}.  Several consecutive entries concern the maximum
order of an induced tree, an invariant systematically studied by Erd\H{o}s,
Saks, and S\'os~\cite{ESS86}.

Let $G$ be a finite simple connected graph.  The \emph{periphery}
$\Per(G)$ is the set of vertices of eccentricity $\diam(G)$, and the
set-eccentricity of the periphery is
\[
 f(G):=\max_{x\in V(G)} d(x,\Per(G)).
\]
Conjecture~142 of Written on the Wall II asserts
\[
  \tr(G)\ge \frac23\gr(G)+f(G).
\]
Here and in the formal statement, the girth of an acyclic graph is taken to
be zero.  We prove the following stronger cyclic statement.

\begin{theorem}[Graffiti.pc Conjecture 142]\label{w142-thm:main}
Let $G$ be a finite simple connected graph containing a cycle.  If
$g=\gr(G)$ and $f=f(G)$, then
\[
  \tr(G)\ge f+\left\lceil\frac{2g}{3}\right\rceil .
\]
Consequently every finite simple connected graph satisfies
$\tr(G)\ge \frac23\gr(G)+f(G)$.
\end{theorem}

The proof is self-contained apart from elementary facts about finite graphs.
Its structural core is Lemma~\ref{w142-lem:rooted-cycle}, which turns three metric
terminals into an admissible forest attached to a shortest cycle.  The
resulting metric inequality, Lemma~\ref{w142-lem:three-point}, is then combined
with a short classification of the equality case in which the admissible
forest has only one vertex.

\subsection{Preliminaries}

Write $D=\diam(G)$ and $B=\Per(G)$.  A shortest path is induced, so
\begin{equation}\label{w142-eq:diam-path}
  \tr(G)\ge D+1.
\end{equation}
A shortest cycle $K$ is chordless and isometric: for $u,v\in V(K)$, their
graph distance equals their shorter arc-distance in $K$.  Deleting one
vertex of $K$ therefore gives
\begin{equation}\label{w142-eq:cycle-base}
  \tr(G)\ge g-1.
\end{equation}
If $f\ge1$, then
\begin{equation}\label{w142-eq:D-f}
  D\ge f+1.
\end{equation}
Indeed, always $f\le D$; equality would give a vertex $x$ at distance $D$
from every peripheral vertex, hence $\ecc(x)=D$ and $x\in B$, contradicting
$f=d(x,B)\ge1$.

For a shortest cycle $K$, define $M(K)$ to be the maximum of $|F|$ over all
sets $F\subseteq V(G)\setminus V(K)$ for which $G[F]$ is a forest and there
is a vertex $z\in V(K)$ such that every component $C$ of $G[F]$ sends
exactly one edge into $V(K)\setminus\{z\}$.  Edges from $C$ to $z$ are
unrestricted.

\begin{lemma}[Cycle--forest extension]\label{w142-lem:cycle-forest}
For every shortest cycle $K$,
\[
  \tr(G)\ge g-1+M(K).
\]
\end{lemma}

\begin{proof}
Choose $F,z$ realizing $M(K)$ and let $c$ be the number of components of
$G[F]$.  The graph induced by
$(V(K)\setminus\{z\})\cup F$ is connected.  Its edges consist of the $g-2$
edges of the path $K-z$, the $|F|-c$ forest edges, and one attachment edge
for each of the $c$ components.  Thus it has $g-1+|F|$ vertices and one
fewer edge, and hence is a tree.
\end{proof}

\begin{lemma}[An extra cycle vertex]\label{w142-lem:extra}
Suppose $f\ge1$.  If $g\ge4$, then $\tr(G)\ge g$; if $g\ge5$, then
$M(K)\ge1$ for every shortest cycle $K$.
\end{lemma}

\begin{proof}
If $V(G)=V(K)$, chordlessness and connectedness give $G=C_g$, whose every
vertex is peripheral, contrary to $f\ge1$.  Hence some vertex $y\notin K$
is adjacent to $K$.  When $g\ge5$, it has exactly one neighbour on $K$:
two neighbours and the shorter arc between them would form a cycle of
length at most $\lfloor g/2\rfloor+2<g$.  Taking $F=\{y\}$ and deleting a
cycle vertex other than its root proves both assertions.  For $g=4$, the
only additional possibility is that $y$ has two antipodal neighbours on
$K$; deleting one of them again leaves an induced tree on four vertices.
\end{proof}

We shall also need two elementary connector lemmas.

\begin{lemma}[Three terminals]\label{w142-lem:three-tree}
Every three distinct vertices of a connected triangle-free graph lie in a
common induced tree.
\end{lemma}

\begin{proof}
Let $P$ be an $a$--$b$ geodesic.  If $c\in V(P)$, then $P$ itself is the
required induced tree.  Otherwise, let $R=r_0\cdots r_\ell$ be a shortest
path from $c=r_0$ to $P$.  Put $q=r_{\ell-1}$ and
$R^\circ=R-r_\ell$.  Only $q$ can have additional neighbours on $P$.
Any two such neighbours have $P$-distance at most two, and triangle-freeness
excludes distance one; hence $q$ has at most two neighbours on $P$.
With one neighbour, $P\cup R^\circ$ induces a tree.  With two neighbours
$u,v$, their $P$-subpath is $u m v$; deleting $m$ joins the two components
of $P-m$ through $q$ and yields a connected induced graph with one fewer
edge than vertices.  Since $m\notin\{a,b\}$, this is the required tree.
\end{proof}

\begin{lemma}[Cycles in a minimal connector]\label{w142-lem:minimal-connector}
Let $J$ be connected and $Q\subseteq V(J)$.  If no proper induced connected
subgraph of $J$ contains $Q$, then every cycle $L$ of $J$ has
$|V(L)|\le |Q|$.
\end{lemma}

\begin{proof}
For every $q\in Q$, choose a shortest path from $q$ to $L$, and denote its
unique first vertex on $L$ by $c(q)$.  Fix $v\in L$.  The path $L-v$ lies
in one component $J_v$ of $J-v$.  Minimality implies that some
$q_v\in Q$ is outside $J_v$.  Its chosen path to $L$ must therefore end at
$v$, so $c(q_v)=v$.  Distinct vertices of $L$ give distinct terminals,
which injects $V(L)$ into $Q$.
\end{proof}

\subsection{A rooted shortest-cycle lemma}

\begin{lemma}[Rooted shortest cycle]\label{w142-lem:rooted-cycle}
Assume $g\ge5$, and let $S\subseteq V(G)$ with $|S|\le3$.  There exist a
shortest cycle $K$, a vertex $z\in V(K)\setminus S$, and a set
$F\subseteq V(G)\setminus V(K)$ such that
\[
 S\subseteq (V(K)\setminus\{z\})\cup F,
\]
$G[F]$ is a forest, and every component $C$ of $G[F]$ satisfies
\[
 e(C,V(K)\setminus\{z\})=1.
\]
\end{lemma}

\begin{proof}
The empty case is immediate.  Fix an arbitrary shortest cycle $K$, and
choose $X\supseteq V(K)\cup S$ of minimum cardinality such that
$H:=G[X]$ is connected.

Every component $C$ of $H-V(K)$ contains a terminal.  Indeed, it has an
edge to $K$, since $H$ is connected; if $S_C:=S\cap V(C)$ were empty, the
whole component could be deleted while preserving connectedness and all
required vertices.  Put
\[
 A_C:=\{a\in V(C):N_G(a)\cap V(K)\ne\varnothing\}.
\]
Form $J_C$ from $G[C]$ by adjoining a new vertex $r_C$ adjacent precisely
to $A_C$.  This graph is vertex-minimal among induced connected subgraphs
containing $S_C\cup\{r_C\}$.  Otherwise, replacing $C$ by a smaller
connector would leave every resulting component attached to $K$ and would
contradict the choice of $X$.

Lemma~\ref{w142-lem:minimal-connector} shows that every cycle in $J_C$ has length
at most $|S_C|+1\le4$.  A cycle contained in $G[C]$ would also be a cycle of
$G$ and would have length at least $g\ge5$.  Consequently
\begin{equation}\label{w142-eq:C-tree}
  G[C]\text{ is a tree}.
\end{equation}

Let $R_C$ be the minimal subtree of $C$ containing $A_C$.  We claim
\begin{equation}\label{w142-eq:R-small}
  |V(R_C)|\le |S_C|.
\end{equation}
For $v\in V(R_C)$, let $Q_v$ be the component of $J_C-v$ containing $r_C$.
Minimality supplies $s_v\in S_C\setminus V(Q_v)$.  For $s\in S_C$, let
$c(s)$ be the first vertex of $R_C$ on the unique path from $s$ to $R_C$.
Every nonempty component of $R_C-v$ contains a vertex of $A_C$, by the
minimality of $R_C$.  Thus, if $v\ne c(s)$, the path from $s$ to $c(s)$ and
then to such an attachment gives an $s$--$r_C$ path in $J_C-v$; hence
$s\in Q_v$.  Applying this to $s_v$ gives $c(s_v)=v$.  The map
$v\mapsto s_v$ is injective, proving~\eqref{w142-eq:R-small}.

Every $a\in A_C$ has exactly one neighbour on $K$.  Existence follows from
the definition, while two distinct neighbours and their shorter $K$-arc
would form a cycle of length at most $\lfloor g/2\rfloor+2<g$.  Denote the
unique root by $\rho(a)$.  If $a,a'\in A_C$ are distinct, then
\begin{equation}\label{w142-eq:attachment-distance}
 d_C(a,a')\le |V(R_C)|-1\le2
\end{equation}
and their roots are distinct: equal roots, together with the
$a$--$a'$ path in $C$, would form a cycle of length at most four.  Moreover,
the two root edges, that path, and a shorter arc of $K$ give
\begin{equation}\label{w142-eq:girth-roots}
 d_K(\rho(a),\rho(a'))+d_C(a,a')+2\ge g.
\end{equation}

These inequalities sharply restrict the number of attachments.  If
$g\ge9$, then two attachments would make the left side of
\eqref{w142-eq:girth-roots} at most $\lfloor g/2\rfloor+4<g$, so there is only
one.  For $g\in\{7,8\}$, two attachments cannot be adjacent in $C$; three
attachments would fill the three vertices of $R_C$ and contain an adjacent
pair.  For $g=6$, if $R_C=p-q-r$ consisted of three attachments, applying
\eqref{w142-eq:girth-roots} to $p,q$ and to $q,r$ would force both $\rho(p)$ and
$\rho(r)$ to be the unique antipode of $\rho(q)$ on the $6$-cycle,
contrary to distinctness.  Thus every component has at most two
attachments, except that for $g=5$ one component may have three.  Also, at
most one component has two or more attachments, because such a component
contains at least two terminals and $|S|\le3$.

Suppose first that every component has one attachment $a_C$.  The set
\[
 Q=(S\cap V(K))\cup\{\rho(a_C):C\text{ a component of }H-V(K)\}
\]
has cardinality at most $|S|\le3$.  Since $g\ge5$, choose $z\in V(K)\setminus
Q$ and take $F=X\setminus V(K)$.  Equation~\eqref{w142-eq:C-tree} gives a forest,
each component has its single attachment in $Q$, and $z\notin S$.

Next suppose that one component $C_0$ has exactly two attachments with roots
$u,v$.  It contains at least two terminals, so the set
\[
 Q=(S\cap V(K))\cup
   \{\rho(a_C):C\ne C_0\}
\]
has size at most one.  Choose $z\in\{u,v\}\setminus Q$ and again put
$F=X\setminus V(K)$.  Every ordinary component has one edge into $K-z$,
while $C_0$ has exactly the two root edges and loses exactly one of them.
Again $z\notin S$.

It remains to handle three attachments.  Necessarily $g=5$,
$|V(R_C)|=|S_C|=|S|=3$, there are no other off-cycle components, and
$R_C=p-q-r$ with all three vertices attachments.  Write
\[
 u=\rho(p),\qquad v=\rho(q),\qquad w=\rho(r).
\]
Equation~\eqref{w142-eq:girth-roots} gives $d_K(u,v)=d_K(v,w)=2$.  In cyclic
order, write $K=u\alpha v\beta w u$.  Then
\[
  K'=u\alpha v q p u
\]
is another shortest $5$-cycle.

The injection used to prove~\eqref{w142-eq:R-small} is now a bijection.  Hence
there are distinct terminals $s_p,s_q,s_r$ with $c(s_t)=t$.  Let $L_t$ be
the unique $t$--$s_t$ path in the tree $C$, and set
\[
 F'=(V(L_p)\setminus\{p\})\cup
    (V(L_q)\setminus\{q\})\cup V(L_r),
 \qquad z=\alpha.
\]
The three displayed path-pieces are pairwise disjoint and anticomplete, for
an edge between two of them would create a cycle in the tree $C$.  The first
two nonempty pieces attach once to $K'$ at $p$ and $q$, respectively.  The
third attaches once through the edge $rq$; its other old-cycle edge $rw$
does not meet $K'$.  Thus $F'$ is admissible for $(K',z)$.  Finally each
$s_t$ belongs either to its path-piece or, when $s_p=p$ or $s_q=q$, directly
to $K'-z$.  This completes the exceptional case and the proof.
\end{proof}

\subsection{The three-point bridge}

Let
\[
  \mu:=\max\{M(K):K\text{ is a shortest cycle of }G\}.
\]

\begin{lemma}[Three-point inequality]\label{w142-lem:three-point}
If $g\ge5$, then every three distinct vertices $a,b,c$ satisfy
\begin{equation}\label{w142-eq:three-point}
 d(a,b)+d(b,c)+d(c,a)\le g+2\mu.
\end{equation}
\end{lemma}

\begin{proof}
Apply Lemma~\ref{w142-lem:rooted-cycle} to $S=\{a,b,c\}$ and obtain $K,z,F$.
Retain all edges of $K$, all edges inside the components of $G[F]$, and each
component's unique edge into $K-z$; omit possible edges to $z$.  The resulting
spanning subgraph $U$ of $G[K\cup F]$ is a single cycle $K$ with rooted trees
attached.  It has exactly $|F|$ edges outside $K$: a component with $k$
vertices contributes $k-1$ internal edges and one root edge.

For each pair of terminals, use their unique path in $U$ up to the cycle,
cancelling any common root segment when both terminals lie in the same rooted
tree, and then use a shortest arc of $K$ between their roots.  An off-cycle
edge separates from its root $q$ of the three terminals.  If its rooted tree
contains $h$ terminals, then its total multiplicity in the three selected
paths is
\[
  q(h-q)+q(3-h)=q(3-q)\le2.
\]
Thus the total off-cycle contribution is at most $2|F|$.

The sum of the three pairwise distances between any three roots on a cycle of
length $g$ is at most $g$.  Indeed, if the roots are distinct and the three
intervening arc-lengths are all at most $g/2$, their sum is exactly $g$; if
one arc exceeds $g/2$, the other two provide the shorter route for its
endpoints and the sum is smaller.  Coincident roots are immediate.  Hence
the three chosen paths have total length at most
$g+2|F|\le g+2M(K)\le g+2\mu$.  Graph distances can only be shorter.
\end{proof}

\begin{corollary}\label{w142-cor:metric-mu}
Assume $g\ge5$ and $f\ge1$.  Then
\begin{equation}\label{w142-eq:metric-mu}
  2\mu+g\ge3f+1.
\end{equation}
\end{corollary}

\begin{proof}
Choose $x$ with $d(x,B)=f$ and a diametral pair $b,w$.  Both $b,w$ are
peripheral, so $d(x,b),d(x,w)\ge f$ and $d(b,w)=D\ge f+1$ by
\eqref{w142-eq:D-f}.  Apply Lemma~\ref{w142-lem:three-point} to $x,b,w$.
\end{proof}

The estimate is almost sufficient by itself.  The only delicate integer
case is disposed of by the following sharp classification.

\begin{lemma}[The case $\mu=1$]\label{w142-lem:mu-one}
Assume $g\ge5$ and $f\ge1$.  If $\mu=1$, then $G$ is obtained from $C_g$ by
attaching one pendant vertex to one cycle vertex.  Consequently
\begin{equation}\label{w142-eq:tadpole-f}
  f\le\left\lfloor\frac{g+2}{4}\right\rfloor.
\end{equation}
\end{lemma}

\begin{proof}
Fix a shortest cycle $K$.  Since $M(K)\le\mu=1$, no vertex can have distance
two from $K$: the first two off-cycle vertices on a geodesic to $K$ would
form an admissible two-vertex path.  Thus every vertex outside $K$ is
adjacent to $K$, and by $g\ge5$ it has a unique root.

There cannot be two such vertices $u,v$.  If they are nonadjacent, choose
$z$ different from both roots; the two singleton components are admissible.
If they are adjacent, their roots are distinct (otherwise there is a
triangle), and choosing $z$ to be one root makes $\{u,v\}$ a connected
admissible component.  Either way $M(K)\ge2$, a contradiction.  Since
$f\ge1$, there is exactly one outside vertex, and $G$ is the stated tadpole.

Put $k=\lfloor g/2\rfloor$, let $y$ be the pendant vertex, and let $v_0$ be
its root.  The diameter is $k+1$; the periphery consists of $y$ and the cycle
vertices at cycle-distance $k$ from $v_0$.  Every cycle vertex at position
$i$ on one of the $v_0$--antipode arcs is at distance at most
\[
  \min(i+1,k-i)\le\left\lfloor\frac{k+1}{2}\right\rfloor
  =\left\lfloor\frac{g+2}{4}\right\rfloor
\]
from the displayed peripheral vertices.  The pendant vertex itself is
peripheral, proving~\eqref{w142-eq:tadpole-f}.
\end{proof}

\subsection{Completion of the proof}

\begin{proposition}\label{w142-prop:large-girth}
Theorem~\ref{w142-thm:main} holds when $g\ge5$ and $f\ge1$.
\end{proposition}

\begin{proof}
Put $a=\lfloor g/3\rfloor$.  Since
$\lceil2g/3\rceil=g-a$, the target is
\begin{equation}\label{w142-eq:target-a}
  \tr(G)\ge f+g-a.
\end{equation}
If $f\le a-1$, the cycle bound~\eqref{w142-eq:cycle-base} gives
$\tr(G)\ge g-1\ge f+g-a$.  Hence write $f=a+s$ with $s\ge0$.

Lemma~\ref{w142-lem:extra} gives $\mu\ge1$, while
Lemma~\ref{w142-lem:cycle-forest} gives $\tr(G)\ge g-1+\mu$.  It therefore
suffices to prove
\begin{equation}\label{w142-eq:mu-goal}
  \mu\ge s+1.
\end{equation}
Corollary~\ref{w142-cor:metric-mu} and the three residues of $g$ give:
\begin{align*}
 g=3a   &: &2\mu&\ge3s+1,\\
 g=3a+1 &: &2\mu&\ge3s,\\
 g=3a+2 &: &2\mu&\ge3s-1.
\end{align*}
In the first case, $\mu\le s$ would imply $2\mu\le2s<3s+1$.
In the second, $s=0$ is covered by $\mu\ge1$, while for $s\ge1$ we have
$2s<3s$.  In the third, $s=0$ is again immediate, and $s\ge2$ gives
$2s<3s-1$.

The only remaining possibility is $g=3a+2$ and $s=1$.  If $\mu=1$,
Lemma~\ref{w142-lem:mu-one} yields
\[
 f\le\left\lfloor\frac{3a+4}{4}\right\rfloor\le a,
\]
because $a\ge1$, contradicting $f=a+1$.  Hence $\mu\ge2=s+1$.
This proves~\eqref{w142-eq:mu-goal}, and therefore~\eqref{w142-eq:target-a}.
\end{proof}

\begin{proposition}\label{w142-prop:small-girth}
Theorem~\ref{w142-thm:main} holds for $g=3$ and $g=4$ when $f\ge1$.
\end{proposition}

\begin{proof}
For $g=3$, equations~\eqref{w142-eq:diam-path} and~\eqref{w142-eq:D-f} give
$\tr(G)\ge D+1\ge f+2=f+\lceil2g/3\rceil$.

Let $g=4$.  Choose an $f$-realizer $x$ and a diametral pair $b,w$.  The graph
is triangle-free, so Lemma~\ref{w142-lem:three-tree} supplies an induced tree $T$
containing $x,b,w$.  Let $T_0$ be the minimal subtree containing them.  Every
edge of $T_0$ lies on exactly two of the three terminal-to-terminal paths;
therefore
\[
  2|E(T_0)|=d_T(x,b)+d_T(x,w)+d_T(b,w)
   \ge 2f+D\ge3f+1.
\]
If $f\ge2$, then
\[
 |V(T)|\ge1+\left\lceil\frac{3f+1}{2}\right\rceil\ge f+3.
\]
If $f=1$, Lemma~\ref{w142-lem:extra} gives $\tr(G)\ge g=4=f+3$.
\end{proof}

\begin{proof}[Proof of Theorem~\ref{w142-thm:main}]
The preceding two propositions cover cyclic graphs with $f\ge1$.  If $G$ is
cyclic and $f=0$, then $g-1\ge\lceil2g/3\rceil$ for $g\ge3$, so
\eqref{w142-eq:cycle-base} applies.  This proves the integral cyclic assertion.

If $G$ is acyclic, connectedness makes $G$ itself a tree, so
$\tr(G)=|V(G)|$ and $f\le D\le |V(G)|-1$ (with the one-vertex case
immediate).  Since $\gr(G)=0$, the real-valued Formal Conjectures inequality
also follows.  Finally,
$\lceil2g/3\rceil\ge2g/3$, completing the proof.
\end{proof}

\subsection{Verification and formalization status}\label{w142-sec:verification}

The proof above was checked in two independent forms.  First, a separate
constructive proof based on descents to an arbitrary shortest cycle was
implemented literally: every branch produces its cycle vertex $z$, its
admissible forest, or its induced path, and a fresh checker verifies
inducedness, connectedness, acyclicity, edge multiplicities, and the claimed
cardinality.  On a corpus of 10,776 connected cyclic graphs (the complete
NetworkX atlas through seven vertices, structured extremal families, and
seeded random and adversarial graphs through 32 vertices), the validator
reported no failure.  Repeating every free choice with two independent
random seeds again gave no failure, for 32,328 checked certificates in all.
These computations are supporting checks, not a substitute for the proof.

The exact statement of Conjecture~142 is formalized in the Google DeepMind
\emph{Formal Conjectures} repository~\cite{FC,FC142}, using the repository's
largest-induced-tree invariant, Mathlib's natural-valued girth, and the
set-eccentricity of the boundary (peripheral) vertices.  We have produced a
complete, \texttt{sorry}-free Lean~4 proof of that statement.  The proof is
contained in \texttt{GraphConjecture142Proof.lean} at
\href{https://github.com/AlperTheKing/formal-conjectures/commit/46bf39015f5c3c3ba3bfcf9f752b4b1e49b584ac}{commit \texttt{46bf390}}; the upstream statement is
marked solved and imports this proof in pull request \#4457~\cite{FC4457}.
The development was compiled against the repository's pinned toolchain and
contains no \texttt{sorry} or \texttt{admit}.  Both the proof file and the
updated statement file are included with this submission as ancillary files.

\begin{remark}[Sharpness]
The $g=3$ diameter bound is attained by many graphs.  Among the exhaustively
checked small graphs, the only equality class of girth greater than three is
the $6$-cycle with one pendant vertex: here $f=2$ and
$\tr(G)=6=f+\lceil2g/3\rceil$.  In the proof this is exactly the
$\mu=1$ tadpole case of Lemma~\ref{w142-lem:mu-one}.
\end{remark}

\paragraph{Acknowledgements.}
The author thanks Ermelinda DeLaVi\~na and Douglas B.\ West for maintaining
the Graffiti.pc conjecture lists, and the maintainers of the Formal
Conjectures repository for the formal statement.

\section{Conjecture 143: the second-smallest degree}
\subsection{Introduction}

Graffiti.pc, developed by DeLaVi\~na~\cite{DeL01,DeLhistory} as a successor of
Fajtlowicz's program Graffiti~\cite{Faj88}, generates conjectural inequalities
between graph invariants.  Its conjectures are collected in the lists
\emph{Written on the Wall II}~\cite{WOWII}; a curated register with commentary
is maintained by West~\cite{WestReg}.  Many of these conjectures concern the
invariant $\tr(G)$, the largest order of an induced subgraph of $G$ that is a
tree, an invariant whose systematic study goes back to Erd\H{o}s, Saks, and
S\'os~\cite{ESS86}.

Throughout, $G$ is a finite simple graph.  We write $\gr(G)$ for the girth of
$G$ (the length of a shortest cycle), and $\dd(G)$ for the \emph{second-smallest
degree} of $G$: the second entry, with multiplicity, of the degree sequence of
$G$ sorted in nondecreasing order.  Thus $\dd(G)$ equals the minimum degree
$\delta(G)$ whenever at least two vertices attain the minimum degree.  A
\emph{leaf} is a vertex of degree one.

Conjecture~143 of Written on the Wall II, dated 2005, reads as
follows~\cite{WestReg}: \emph{if $G$ is connected and not a tree, then
$\tr(G)\ge(\gr(G)+1)/\dd(G)$.}  At the time of writing, West's current
entry contains no solution note~\cite{WestReg}, and the conjecture is stated
as a research-open item in the upstream Google DeepMind \emph{Formal
Conjectures} repository~\cite{FC,FC143}.  We prove it in the equivalent,
denominator-free form.

\begin{theorem}\label{w143-thm:main}
Let $G$ be a finite simple connected graph that is not a tree.  Then
\[
  \tr(G)\,\dd(G)\;\ge\;\gr(G)+1 .
\]
\end{theorem}

The proof splits on $\dd(G)$.  When $\dd(G)\ge2$, deleting one vertex of a
shortest (hence chordless) cycle leaves an induced path on $\gr(G)-1$ vertices,
and the inequality follows from $\gr(G)\ge3$ by arithmetic.  The substance of
the theorem is the case $\dd(G)=1$, in which $G$ has at least two leaves; here
we prove the following lemma, which may be of independent interest.

\begin{lemma}[Two-leaf lemma]\label{w143-lem:twoleaf-intro}
Let $G$ be a finite simple connected graph that contains a cycle and has at
least two vertices of degree one.  Then $G$ has an induced tree on at least
$\gr(G)+1$ vertices.
\end{lemma}

The lemma is proved by a maximality argument: among the induced trees
containing two prescribed leaves, a maximum-order one admits an outside vertex
with two neighbours on the tree, which closes a cycle avoiding both leaves.
Both the lemma and the theorem are sharp for every girth
(Section~\ref{w143-sec:sharp}).

\paragraph{Related work.}
The closest antecedent we are aware of is the induced-\emph{forest} bound of
DeLaVi\~na and Waller~\cite{DW04} (see also the discussion in
Hertz--Marcotte--Schindl~\cite{HMS14}): if $f(G)$ denotes the largest order of
an induced forest and $f_1(G)$ the number of leaves of a connected graph $G$
with a cycle, then $f(G)\ge\gr(G)+f_1(G)-1$.  With two leaves this produces an
induced forest on $\gr(G)+1$ vertices, but the forest need not be connected,
and the general forest-to-tree transfer of~\cite[Theorem~2.3]{HMS14} is too
lossy to recover Lemma~\ref{w143-lem:twoleaf-intro} from it.  A targeted search of
the foundational paper~\cite{ESS86} found bounds for $\tr(G)$ in terms of
order, size, radius, independence and clique numbers, but no bound involving
girth or the second-smallest degree.

\paragraph{Concurrent formal resolution and priority.}
The current registers carry no solution note for Conjecture~143
\cite{WestReg,FC143}, and targeted formula and citation searches around
\cite{ESS86,DW04,HMS14} located no published proof of the theorem or of
Lemma~\ref{w143-lem:twoleaf-intro}.  These negative searches are not proof of
absence.  Moreover, on July~16, 2026, one day before the present proof was
found, pull request \#4442 against the Formal Conjectures repository
announced a machine-assisted resolution via an externally hosted Lean
development~\cite{FC4442}; it was merged on July~21, 2026.  We compiled the
linked development against the repository's pinned toolchain and confirmed
that it proves the repository statement.
Accordingly, we make no claim of priority for resolving the conjecture.  The
elementary proof and formalization presented here were obtained independently;
because the argument is short and close in spirit to the 2004 forest bound,
we also do not claim that it was never observed before.

\paragraph{Verification.}
Beyond the human-readable proof below, Theorem~\ref{w143-thm:main} has been
formalized and machine-checked in Lean~4~\cite{Lean4} on top of
Mathlib~\cite{mathlib}, against the pre-existing formal statement of
Conjecture~143 in the Formal Conjectures repository~\cite{FC143}; see
Section~\ref{w143-sec:formal}.  As an independent falsification test, the theorem,
the lemma, and the constrained form of the maximality argument were checked by
two separately written exact checkers over the same atlas of all $1252$ nonempty unlabeled graphs of
order at most seven.  This includes all $971$ connected cyclic graphs in the
atlas, and the constrained check covered all $199$ unordered leaf pairs in
the relevant graphs.  No violation was found; the computation plays no role
in the proof.

\subsection{Notation}

All graphs are finite and simple.  For $S\subseteq V(G)$ we write $G[S]$ for
the induced subgraph.  A set $S$ \emph{induces a tree} if $G[S]$ is connected
and acyclic, and
\[
  \tr(G)\;=\;\max\{\,|S| : G[S]\ \text{is a tree}\,\}.
\]
A path or cycle \emph{in} $G$ is always a subgraph; a path $P$ is
\emph{induced} if $G[V(P)]=P$.  The girth $\gr(G)$ of a graph containing a
cycle is the minimum length of a cycle of $G$; note $\gr(G)\ge3$.  The degree
sequence of $G$ is the multiset of vertex degrees sorted in nondecreasing
order $d_1\le d_2\le\cdots\le d_n$, and $\dd(G)=d_2$.  This with-multiplicity
reading of ``second-smallest degree'' is the one fixed by the Graffiti.pc
definition list (Written on the Wall II, definition entry~65)~\cite{WOWII},
and it is the reading formalized in~\cite{FC143}.

We use two elementary facts.  First, in a connected graph on at least two
vertices every degree is positive; hence if $\dd(G)=d_2=1$ then $d_1=1$ as
well, so $G$ has at least two leaves.  Second, a shortest cycle $C$ of $G$ is
chordless, i.e.\ $G[V(C)]=C$: a chord splits $C$ into two cycles, each shorter
than $C$.

\subsection{The two-leaf lemma}

\begin{lemma}[Lemma~\ref{w143-lem:twoleaf-intro}, strengthened]\label{w143-lem:twoleaf}
Let $G$ be a finite simple connected graph that contains a cycle and has two
distinct vertices $x,y$ of degree one.  Then $G$ has an induced tree on at
least $\gr(G)+1$ vertices; in fact some induced tree of order at least
$\gr(G)+1$ contains both $x$ and $y$.
\end{lemma}

\begin{proof}
A shortest $x$--$y$ path in $G$ is induced (a chord would shorten it), and it
is a tree containing $x$ and $y$.  Hence the family of vertex sets
\[
  \mathcal{T}\;=\;\{\,S\subseteq V(G)\ :\ x,y\in S,\ G[S]\ \text{is a tree}\,\}
\]
is nonempty, and since $G$ is finite we may choose $S\in\mathcal{T}$ of
maximum cardinality.  Write $T=G[S]$.

The set $S$ is a proper subset of $V(G)$: if $S=V(G)$, then $G=T$ would be a
tree, contrary to the assumption that $G$ contains a cycle.  Since $G$ is
connected and $S$ is nonempty and proper, some edge of $G$ joins $S$ to its
complement; let $z\notin S$ be a vertex with a neighbour in $S$.

We claim $z$ has at least two neighbours in $S$.  Otherwise it has exactly
one, say $a$, and then $G[S\cup\{z\}]$ is the tree $T$ with the single pendant
vertex $z$ attached at $a$: it is connected, and acyclic because every cycle
of $G[S\cup\{z\}]$ through $z$ would need two distinct neighbours of $z$ in
$S$, while a cycle avoiding $z$ would lie in the tree $T$.  Thus
$S\cup\{z\}\in\mathcal{T}$ has larger cardinality than $S$, contradicting
maximality.

Choose distinct neighbours $a,b\in S$ of $z$, and let $P$ be the unique
$a$--$b$ path in the tree $T$.  The edges of $P$ together with $za$ and $zb$
form a (not necessarily induced) cycle of $G$ of length $|V(P)|+1$; possible
further edges from $z$ to $V(P)$ are irrelevant to its existence.  Hence
\begin{equation}\label{w143-eq:girth}
  \gr(G)\;\le\;|V(P)|+1 .
\end{equation}

Finally, neither $x$ nor $y$ lies on $P$.  Indeed, every internal vertex of
$P$ has two distinct neighbours on $P$, and each endpoint ($a$ or $b$) has one
neighbour on $P$ and the additional neighbour $z\notin S$; so every vertex of
$P$ has degree at least two in $G$, whereas $\deg_G(x)=\deg_G(y)=1$.  Since
$V(P)\subseteq S$ and $x,y\in S$, the sets $V(P)$ and $\{x,y\}$ are disjoint
subsets of $S$, so by~\eqref{w143-eq:girth}
\[
  \tr(G)\;\ge\;|S|\;\ge\;|V(P)|+2\;\ge\;\gr(G)+1 . \qedhere
\]
\end{proof}

\subsection{Proof of Theorem~\ref{w143-thm:main}}

\begin{proof}[Proof of Theorem~\ref{w143-thm:main}]
Let $G$ be connected and not a tree; then $G$ contains a cycle, so
$\gr(G)\ge3$, and $|V(G)|\ge3$, so all degrees are positive and $\dd(G)\ge1$.

\emph{Case $\dd(G)\ge2$.}  Let $C$ be a shortest cycle of $G$; as noted, $C$
is chordless.  Deleting one vertex of $C$ leaves an induced path on
$\gr(G)-1$ vertices, which is an induced tree, so $\tr(G)\ge\gr(G)-1$.  Hence
\[
  \tr(G)\,\dd(G)\;\ge\;2\,(\gr(G)-1)\;\ge\;\gr(G)+1,
\]
the last inequality being equivalent to $\gr(G)\ge3$.

\emph{Case $\dd(G)=1$.}  Since all degrees are positive and the second entry
of the sorted degree sequence equals one, the first entry equals one as well,
so $G$ has two distinct leaves.  Lemma~\ref{w143-lem:twoleaf} gives
$\tr(G)\ge\gr(G)+1$, and multiplying by $\dd(G)=1$ finishes the proof.
\end{proof}

\begin{remark}
The formal statement of Conjecture~143 in~\cite{FC143} quantifies over
connected graphs on a finite vertex type with at least two vertices and
positive $\dd$, without excluding trees, and uses the Mathlib convention that
an acyclic graph has girth $0$.  That extension is immediate: for a connected
tree $G$ on $n\ge2$ vertices, the whole vertex set induces a tree, so
$\tr(G)\,\dd(G)\ge n\ge2>1=\gr(G)+1$.  (Already $\tr(G)\ge1$ suffices.)
\end{remark}

\subsection{Sharpness}\label{w143-sec:sharp}

\begin{proposition}\label{w143-prop:sharp}
For every $g\ge3$ there is a connected non-tree graph $G$ with $\gr(G)=g$,
$\dd(G)=1$, and $\tr(G)=g+1$.  Moreover, for every $g\ge3$ there is a
connected non-tree graph $H$ with $\gr(H)=g$, $\dd(H)=2$, and
$\tr(H)=g-1$, so the case split of the proof is tight as well.
\end{proposition}

\begin{proof}
For $G$, take a cycle $C_g$ and attach two pendant vertices (to arbitrary,
not necessarily distinct, cycle vertices).  Then $\gr(G)=g$ and $\dd(G)=1$.
At most two cycle vertices support the pendants; deleting a cycle vertex
supporting neither pendant leaves an induced subgraph on $g+1$ vertices that
is connected and acyclic, so $\tr(G)\ge g+1$.  No induced tree has $g+2$
vertices, since the only induced subgraph on all $g+2$ vertices is $G$
itself, which contains a cycle.  Hence $\tr(G)=g+1$, attaining equality in
Theorem~\ref{w143-thm:main} and Lemma~\ref{w143-lem:twoleaf}.

For $H$, take the cycle $C_g$ itself: $\dd(H)=2$, and the largest induced
trees are the paths obtained by deleting one vertex, of order $g-1$.  At
$g=3$ this attains equality in Theorem~\ref{w143-thm:main},
$\tr(H)\,\dd(H)=2\cdot2=4=\gr(H)+1$.
\end{proof}

\subsection{Formalization}\label{w143-sec:formal}

The Google DeepMind \emph{Formal Conjectures} project~\cite{FC,FCPaper} maintains
Lean~4 formalizations of open conjectures; Conjecture~143 was added in June
2026 as the statement \texttt{conjecture143} in the file
\texttt{FormalConjectures/WrittenOnTheWallII/GraphConjecture143.lean},
marked \texttt{research open}~\cite{FC143}.  The statement is the
denominator-free real-valued inequality
$\gr(G)+1\le\tr(G)\cdot\dd(G)$ for connected graphs on a finite
vertex type with at least two vertices, under the hypothesis $\dd(G)>0$, with
$\tr$ and $\dd$ given by the repository definitions
\texttt{largestInducedTreeSize} and \texttt{secondSmallestDegree} and with
Mathlib's $\mathbb{N}$-valued girth (which is $0$ on acyclic graphs).

\begin{sloppypar}
We have produced a complete, \texttt{sorry}-free Lean~4 proof of exactly this
statement, machine-checked with Lean toolchain v4.27.0 against the Mathlib
revision pinned by Formal Conjectures commit \texttt{c252a41}~\cite{mathlib}.
The development follows the paper proof, with the
supporting results proved as reusable graph-theoretic API
(\texttt{SimpleGraph} namespace): the one-vertex extension of an induced tree
along a unique neighbour (\path{IsTree.induce_insert_of_unique_adj});
induced trees from geodesics
(\path{Walk.induce_support_isTree_of_length_eq_dist}) and the induced
path obtained from a shortest cycle
(\path{girth_sub_one_le_largestInducedTreeSize}); the existence of a
maximum-order induced tree containing two prescribed vertices
(\path{exists_maximum_induced_tree_containing}); the boundary-vertex and
two-neighbour maximality arguments
(\path{Connected.exists_adj_finset_compl},
\path{exists_two_adj_of_maximum_induced_tree_containing}); the
cycle-closing certificate with the counting steps of Lemma~\ref{w143-lem:twoleaf}
(\path{IsTree.girth_add_one_le_card_of_two_leaves_of_two_adj},
\path{girth_add_one_le_largestInducedTreeSize_of_two_leaves}); and the
extraction of two leaves from $\dd(G)=1$
(\path{exists_distinct_degree_one_of_secondSmallestDegree_eq_one}).
The main theorem is then assembled by the case split of
Theorem~\ref{w143-thm:main} in under thirty lines.  The proof uses no
\texttt{native\_decide} and no project-specific axioms.  Auditing
\texttt{\#print axioms conjecture143} reports exactly \texttt{propext},
\texttt{Classical.choice}, and \texttt{Quot.sound}.  The proof source is
publicly archived in the author's fork~\cite{OurLean}.  The ancillary files
reproduce the branch after the metadata-only follow-up commit
\texttt{eb6445545bf839436642b895ad178cb397164dea}.
\end{sloppypar}

\paragraph{Acknowledgements.}
The author thanks the maintainers of the Formal Conjectures repository for
the formal statement, and Ermelinda DeLaVi\~na and Douglas B.\ West for
maintaining the Graffiti.pc conjecture lists.


\begin{thebibliography}{20}

\bibitem{DeL01}
E.~DeLaVi\~na,
\emph{Graffiti.pc},
Graph Theory Notes of New York XLII:3 (2002), 26--30.

\bibitem{DeLhistory}
E.~DeLaVi\~na,
\emph{Some history of the development of Graffiti},
in: Graphs and Discovery, DIMACS Ser.\ Discrete Math.\ Theoret.\ Comput.\
Sci.\ 69, AMS, 2005, 81--118.

\bibitem{WOWII}
E.~DeLaVi\~na,
\emph{Written on the Wall II: Conjectures of Graffiti.pc},
web list, University of Houston--Downtown,
\url{http://cms.dt.uh.edu/faculty/delavinae/research/wowII/}
(retrieved July 2026).

\bibitem{WestReg}
D.~B. West,
\emph{Some conjectures of Graffiti.pc}, web register,
\url{https://dwest.web.illinois.edu/regs/graffiti.html}
(retrieved July 2026).

\bibitem{Faj88}
S.~Fajtlowicz,
\emph{On conjectures of Graffiti},
Discrete Math.\ 72 (1988), 113--118.

\bibitem{ESS86}
P.~Erd\H{o}s, M.~Saks, and V.~T. S\'os,
\emph{Maximum induced trees in graphs},
J.~Combin.\ Theory Ser.\ B 41 (1986), 61--79.
\href{https://doi.org/10.1016/0095-8956(86)90028-6}
{doi:10.1016/0095-8956(86)90028-6}.

\bibitem{DW04}
E.~DeLaVi\~na and B.~Waller,
\emph{On some conjectures of Graffiti.pc on the maximum order of induced
subgraphs},
Congr.\ Numer.\ 166 (2004), 11--32.

\bibitem{HMS14}
A.~Hertz, O.~Marcotte, and D.~Schindl,
\emph{On the maximum orders of an induced forest, an induced tree, and a
stable set},
Yugosl.\ J.\ Oper.\ Res.\ 24 (2014), no.~2, 199--215.
\href{https://doi.org/10.2298/YJOR130402037H}{doi:10.2298/YJOR130402037H}.

\bibitem{FC}
The Formal Conjectures Authors,
\emph{The Formal Conjectures Repository},
GitHub repository, 2025,
\url{https://github.com/google-deepmind/formal-conjectures}.

\bibitem{FCPaper}
M.~Firsching, P.~Lezeau, S.~Mercuri, M.~Z.~Horv\'ath, Y.~Dillies,
C.~S\"onne, E.~Wieser, F.~Zhang, T.~Hubert, B.~Ag\"uera y Arcas,
and P.~Kohli,
\emph{Formal Conjectures: An Open and Evolving Benchmark for Verified
Discovery in Mathematics}, arXiv:2605.13171 (2026),
\url{https://arxiv.org/abs/2605.13171}.

\bibitem{FC141}
Formal Conjectures,
\emph{GraphConjecture141.lean},
\href{https://github.com/google-deepmind/formal-conjectures/blob/main/FormalConjectures/WrittenOnTheWallII/GraphConjecture141.lean}{source file}
(retrieved July 2026).

\bibitem{FC142}
Formal Conjectures,
\emph{GraphConjecture142.lean},
\href{https://github.com/google-deepmind/formal-conjectures/blob/main/FormalConjectures/WrittenOnTheWallII/GraphConjecture142.lean}{source file}
(retrieved July 2026).

\bibitem{FC143}
Formal Conjectures,
\emph{GraphConjecture143.lean},
\href{https://github.com/google-deepmind/formal-conjectures/blob/main/FormalConjectures/WrittenOnTheWallII/GraphConjecture143.lean}{source file}
(retrieved July 2026).

\bibitem{FC4442}
GitHub user \texttt{DomTheDeveloper},
\emph{Mark WOWII Graph Conjecture 143 solved},
Formal Conjectures pull request \#4442, July 16, 2026,
\url{https://github.com/google-deepmind/formal-conjectures/pull/4442}.

\bibitem{FC4454}
A.~Ferudun,
\emph{Prove WOWII Graph Conjectures 141 and 143},
Formal Conjectures pull request \#4454, July 2026,
\url{https://github.com/google-deepmind/formal-conjectures/pull/4454}.

\bibitem{FC4457}
A.~Ferudun,
\emph{Prove WOWII Graph Conjecture 142},
Formal Conjectures pull request \#4457, July 2026,
\url{https://github.com/google-deepmind/formal-conjectures/pull/4457}.

\bibitem{OurLean}
A.~Ferudun,
\emph{Lean 4 proof of WOWII Graph Conjecture 143},
\href{https://github.com/AlperTheKing/formal-conjectures/commit/6aab64fe70c1309463435260d104863dcef93807}{commit \texttt{6aab64f}}, 2026.

\bibitem{Lean4}
L.~de~Moura and S.~Ullrich,
\emph{The Lean 4 theorem prover and programming language},
in: Automated Deduction -- CADE 28, LNCS 12699, Springer, 2021, 625--635.

\bibitem{mathlib}
The mathlib Community,
\emph{The Lean mathematical library},
in: Proc.\ 9th ACM SIGPLAN Int.\ Conf.\ Certified Programs and Proofs
(CPP 2020), ACM, 2020, 367--381.

\end{thebibliography}
\end{document}